\documentclass[a4paper,12pt]{article}

\usepackage[left=2cm,right=2cm, top=2cm,bottom=2cm,bindingoffset=0cm]{geometry}

\usepackage[russian,english]{babel}
\usepackage{amsthm}
\usepackage{amsmath}
\usepackage{amssymb}
\usepackage{cite}

\usepackage{graphicx}
\usepackage{xcolor}

\newtheorem{thm}{Theorem}
\newtheorem{lemma}{Lemma}
\theoremstyle{definition}
\newtheorem{statement}{Proposition}
\newtheorem{ip}{Inverse problem}
\newtheorem{alg}{Algorithm}

\begin{document}

\begin{center} \bf \Large
Inverse problem for Sturm--Liouville operators  \\
with frozen argument on closed sets
\end{center}

\begin{center}
Maria Kuznetsova \\
Saratov State University (SSU) \\
email: kuznetsovama@info.sgu.ru
\end{center}

{\bf Abstract.}
In the paper, we study the problem of recovering the potential from the spectrum of the Dirichlet boundary value problem for a Sturm--Liouville equation with frozen argument on a closed set.
We consider the case when the closed set consists of two segments and the frozen argument is at the end of the first segment.
A uniqueness theorem and an algorithm solving the inverse problem are obtained along with necessary and sufficient conditions of its solvability.
The considered case  significantly differs from the one of the classical Sturm--Liouville operator with frozen argument.



\medskip

{\color{blue} This is a translation of the paper published in russian: \\
M. A. Kuznetsova, "Inverse problem for the Sturm--Liouville operator with a frozen argument on the time scale",  Itogi Nauki i Tekhniki. Ser. Sovrem. Mat. Pril. Temat. Obz. {\bf 208} (2022), pp. 49--62.
{https://doi.org/10.36535/0233-6723-2022-208-49-62}
}
%
%

\section{Introduction}
The aim of the paper is to study an inverse spectral problem for a Sturm--Liouville operator with frozen argument on a closed set $T \subseteq \mathbb R.$
Differential operators on closed sets generalize classical differential and difference operators since they contain $\Delta$-derivative, see~\cite{[1], [2], Hilger90}.

Recently, interest has appeared in inverse spectral problems for differential operators on closed sets, see~\cite{ozkan1, ozkan2, yurko2019, y2-structure, uniqueness, ozkan3, mathnotes}. Such problems consist in recovering operators from their spectral characteristics. The most complete results in this direction were obtained for the classical Sturm--Liouville operator on an interval, see~\cite{[3], [4], [5]}.
Posing and studying inverse problems substantially depend on the structure of the considered closed set, which leads to the need for some or other restrictions.
By the present time, the most general form of the sets on which the solution of an inverse problem has been obtained is a finite union of segments and isolated points, see~\cite{uniqueness,mathnotes}.

Inverse problems for the Sturm--Liouville operator with frozen argument on an interval were studied in the series of the works~\cite{vasiliev, BondButVas, wang-froz, frozen, albevirio, nizh, bond-nodal}.
This operator is determined by the Sturm--Liouville differential expression with frozen argument
\begin{equation*}\label{1.1}
\ell y(x) :=-y''(x)+q(x)y(\gamma), \quad 0<x<r,
\end{equation*}
where $\gamma \in [0, r]$ is fixed.
Unlike the classical Sturm--Liouville operator, operators with frozen argument are non-local. For this reason, the methods of the classical inverse problems theory~\cite{[3], [4], [5]} are inapplicable for them. Meanwhile, non-local operators have applications in many areas of mathematics and natural science, see~\cite{loaded, myshkis, hale}.

The problem of recovering the potential $q$ from the spectrum of the boundary value problem
\begin{equation*}\label{1.2}
\ell y = \lambda y, \quad y^{(\alpha)}(0)=y^{(\beta)}(r)=0, \quad \alpha, \beta \in \{0, 1\},
\end{equation*}
was investigated in the works~\cite{vasiliev, BondButVas, wang-froz, frozen}.
In~\cite{frozen}, the case of arbitrary $\gamma/r \in {\mathbb Q}$ was studied and a complete description of the so-called non-degenerate and degenerate cases was given with respect to the values of the three parameters $\gamma / r,$ $\alpha $ and $ \beta.$
In particular, the Dirichlet boundary conditions ($\alpha=\beta=0$) correspond to the degenerate case for any $\gamma/r \in {\mathbb Q}.$
In the non-degenerate case, the potential is uniquely recovered from the spectrum, while in the degenerate one, additional information is required for uniqueness of recovery, e.g., specification of $q$ on a part of the interval.
Necessary and sufficient conditions of the inverse problem solvability were also obtained. In the non-degenerate case, the latter ones include only the spectrum asymptotics
\begin{equation*}
\lambda_n = \frac{\pi^2}{r^2}\bigg(n - \frac{\alpha+\beta}{2} + \frac{\kappa_n}{n}\bigg)^2, \quad n \ge 1, \quad \{ \kappa_n \}_{n \ge 1} \in l_2.
\end{equation*}
In the degenerate case, the condition is added for the coincidence of some eigenvalues infinite part with eigenvalues of the corresponding operator with zero potential.
For the irrational case $\gamma/r \notin {\mathbb Q},$  in the work~\cite{wang-froz}, uniqueness of the potential recovery from the spectrum was proved for any $\alpha$ and $\beta.$
In the matter of inverse spectral problems for operators with frozen argument on closed sets, they have not been considered previously.

In this paper, we consider the boundary value problem for the Sturm--Liouville equation with frozen argument on the closed set $T$ of a special form:
\begin{equation}
-y^{\Delta\Delta}(t) + y(\gamma) q(t) = \lambda y(\sigma(t)),  \quad t \in T,\label{*}\end{equation}
\begin{equation} y(0) = y(b) = 0,\label{Dirichlet}\end{equation}
where
\begin{equation} \label{T}
T=[0, \gamma] \cup [a, b], \quad d := a - \gamma, \; l := b - a. \end{equation}
Structure~\eqref{T} belongs to the simplest ones that allow detecting considerable differences from the case of a segment.

In the paper, we study recovering the potential $q \in C(T)$ from the spectrum of boundary value problem~\eqref{*}--\eqref{Dirichlet}.
We establish the conditions on the quantities $d,$ $\gamma,$ and $l$ under which a uniqueness theorem for the inverse problem solution holds, see Theorem~\ref{uniqueness}.
In particular, uniqueness of recovery holds if $l = k \gamma,$ $k \in \mathbb N.$ Here, there is a difference from the case of the equation with frozen argument on a segment, i.e. for $\gamma = a$ and $r = a + l,$ in which the potential can not be uniquely reconstructed from the spectrum of the Dirichlet boundary value problem for any rational $k.$
As well we obltain an algorithm for recovering the potential along with necessary and sufficient conditions of inverse problem solvability, see Algorithm~\ref{algorithm} and Theorem~\ref{conditions}.
Special form~\eqref{2} of the characteristic function considerably complicates the study as compared with the case of a segment.  In particular, the spectrum characterization is not limited to asymptotics alone, in contrast to the non-degenerate case of the operator on an interval.

\section{Inverse problem statement. A uniqueness theorem}
For convenience of the reader, we first provide necessary notions from the theory of differential equations with $\Delta$-derivative~\cite{[1],[2]}. Let $T$ be a closed subset of the real line.
Define on $T$ the jump functions $\sigma$ and $\sigma_-$  in the following way:
$$
\sigma(x):=\left\{\begin{array}{cl}\inf \{s\in T:\; s>x\}, & x\ne \max T,\\[2mm]
\max T, & x=\max T,
\end{array}\right.
$$
$$
\sigma_-(x):=\left\{\begin{array}{cl}\sup \{s\in T:\; s<x\}, & x\ne \min T,\\[2mm]
\min T, & x=\min T.
\end{array}\right.
$$
A point $x\in T$ is called left-dense, left-isolated, right-dense, and right-isolated, if $\sigma_-(x)=x,$
$\sigma_-(x)<x,$ $\sigma(x)=x,$ and $\sigma(x)>x,$ respectively. If $\sigma_-(x)<x<\sigma(x),$ then $x$ is called {\it isolated}; if
$\sigma_-(x)=x=\sigma(x),$ then $x$ is called {\it dense}.

Denote $T^0:=T\setminus\{\max T\},$ if $\max T$  exists and is left-isolated, and $T^0:=T,$ otherwise. We also denote by $C(B)$ the class of
functions continuous on a subset $B \subseteq T.$

A function $f$ on $T$ is called $\Delta$-differentiable at $t \in T^0,$ if for any $\varepsilon>0$ there exists $\delta>0$ such that
$$
|f(\sigma(t))-f(s)-f^{\Delta}(t)(\sigma(t)-s)|\le\varepsilon |\sigma(t)-s|
$$
for all $s\in (t-\delta, t+\delta)\cap T.$ The value $f^{\Delta}(t)$ is called the $\Delta$-derivative of the function $f$ at the point~$t.$

Consider equation~\eqref{*} on closed set~\eqref{T} with a continuous potential
$q \in C(T).$ Note that for $T$ of form~\eqref{T} we have $T^0 = T.$
We denote by $C_\Delta^n(T),$ $ n \in \mathbb{N},$ the class of
functions having the $n$-th continuous on $T$ $\Delta$-derivative.
Define solutions as functions $y \in C_\Delta^2(T)$ for which identity~\eqref{*} holds. Each $\lambda$ for which there exist a non-zero solution $y$ satisfying the Dirichlet conditions~\eqref {Dirichlet} is called an eigenvalue.

Since the closed set $T$ has form~\eqref{T}, for any $ f \in C_\Delta^1(T),$
we have
\begin{equation} \label{derivative}
f^\Delta(t) =
\left\{\begin{array}{cc} \displaystyle \frac{f(a) - f(\gamma)}{d}, & t = \gamma,\\[2mm]
f'(t),  & t \in [0, \gamma]\cup[a, b],
\end{array}\right.\end{equation}
see~\cite{yurko2019, y2-structure}.
In~\eqref{derivative}, the classical derivative $f '(t)$ exists, and the equality $f'(\gamma) = f^\Delta(\gamma)$ holds due to continuity.
Applying formula~\eqref{derivative} to a solution $y$ of equation~\eqref{*} and its $\Delta$-derivative, we obtain that~\eqref{*} is equivalent to the system of equations
\begin{equation}
\left\{
\begin{array}{c}
-y''(x_1) + q(x_1) y(\gamma) = \lambda y(x_1), \quad x_1 \in [0, \gamma]; \\
-y''(x_2) + q(x_2) y(\gamma) = \lambda y(x_2), \quad x_2 \in [a, b],
\end{array}\right.
\label{system}
\end{equation}
with the jump conditions
\begin{equation} \label{jump conditions}
\begin{pmatrix}
y(a) \\
y'(a)
\end{pmatrix} =\begin{pmatrix}
 1 & d \\
d (q(\gamma) - \lambda ) & 1 - \lambda d^2
\end{pmatrix} \begin{pmatrix}
y(\gamma) \\
y'(\gamma)
\end{pmatrix}.
\end{equation}
From~\eqref{derivative} it also follows that $y \in C_\Delta^2(T)$ is equivalent to the condition
$y \in C^2[0, \gamma] \cap C^2[a, b],$
where $C^2(B)$ denotes the class of functions having classical derivatives of the second order that are continuous on $B.$

Introduce the solutions $S(x, \lambda)$ and $C(x, \lambda)$ of the first equation in \eqref{system} under the initial conditions
\begin{equation*}S(\gamma, \lambda) =0, \; S'(\gamma, \lambda) = 1; \quad  C(\gamma, \lambda) =1, \; C'(\gamma, \lambda) = 0.\end{equation*}
For these functions, the following formulae are known (see~\cite{BondButVas}):
\begin{equation} \label{SC}
S(x, \lambda) = \frac{\sin \rho(x-\gamma)}{\rho}, \; C(x, \lambda) = \cos \rho(x - \gamma) + \int_\gamma^x \frac{\sin \rho (x-t)}{\rho} q(t) \, dt, \quad x \in [0, \gamma];
\end{equation}
here and below $\lambda=\rho^2.$

Any solution of system~\eqref{system}--\eqref{jump conditions} on $[0, \gamma]$ can be represented in the form
\begin{equation} \label{rep1}
y(x) =  AS(x, \lambda)+BC(x, \lambda).
\end{equation}
Taking into account~\eqref{jump conditions}, for $x\in[a, b],$ we have
\begin{multline} \label{rep2}
y(x) = \big(y(\gamma) + d y'(\gamma)\big) \cos \rho(x-a) + \\
+\big( (1-d^2 \lambda) y'(\gamma) -d \lambda y(\gamma) + d y(\gamma) q(\gamma)\big) \frac{\sin \rho (x-a)}{\rho} +
B \int_{a}^x \frac{\sin \rho(x-t)}{\rho} q(t)\, dt.\end{multline}
Substituting representations~\eqref{rep1} and~\eqref{rep2} into boundary conditions~\eqref{Dirichlet}, we obtain the following system of linear equations with respect to $A$ and $B:$
\begin{equation*} \left\{
\begin{array}{c} \\
A S(0, \lambda) + B C(0, \lambda) =0, \\[3mm]
 \displaystyle
A \bigg( d \cos \rho l +  \big[1-d^2 \lambda\big]\frac{\sin \rho l}{\rho} \bigg) +
B \bigg(\cos \rho l + \big[dq(\gamma)-d \lambda\big] \frac{\sin \rho l}{\rho}
+ \int_{a}^{b} \frac{\sin \rho(b - t)}{\rho} q(t) \, dt \bigg)  =0.
\end{array}
\right.
\label{sys1}
\end{equation*}
A number $\lambda$ is an eigenvalue of boundary value problem~\eqref{*}--\eqref{Dirichlet} if and only if there exists a non-zero solution of this system. The determinant of the system $\Delta(\lambda)$ is called the characteristic function of  boundary value problem~\eqref{*}--\eqref{Dirichlet}, and it is an entire function of order $1/2.$
The spectrum of boundary value problem~\eqref{*}--\eqref{Dirichlet} is the sequence of the characteristic function zeros $\{\lambda_n\}_{n \ge 0},$ taking into account the multiplicity.

Using formulae~\eqref{SC}, we get
\begin{equation}\begin{split}
\Delta(\lambda) = -c_1(\lambda) \frac{\sin \rho \gamma}{\rho} - c_2(\lambda) \cos \rho \gamma - \frac{\sin \rho \gamma}{\rho} \int_{a}^{b} \frac{\sin \rho (b-t)}{\rho} q(t)\, dt
\\
-c_2(\lambda) \int_0^{\gamma} \frac{\sin \rho t}{\rho} q(t) \, dt -d q(\gamma)\frac{\sin^2 \rho \gamma}{\rho^2},\end{split}
\label{Delta}\end{equation}
where
\begin{equation*}c_1(\lambda) := \cos \rho l - d \rho \sin \rho l, \quad c_2(\lambda) := d \cos \rho l + \frac{1-d^2 \lambda}{\rho} \sin \rho l.
\end{equation*}

Further, we need the following lemma.
\begin{lemma} \label{riesz}
Positive zeros of the function $c_2(z^2)$ can be selected to a sequence $\{ z_n\}_{n \ge 1}$ such that
the systems $\{ \sin z_n t\}_{n \ge 1}$ and $\{ 1\} \cup \{ \cos z_n t\}_{n \ge 1}$ are Riesz basises in $L_2(0, l).$
\end{lemma}
\begin{proof}
Consider the equation
\begin{equation*} g(z) := \frac{d z}{d^2 z^2 - 1} - \tan z l = 0.
\label{eq}\end{equation*}
If $z$ is a non-zero root of this equation, then the number $z$ makes $c_2(z^2)$ vanish. Note that the function $g(z)$ is continuous and monotone on any interval not containing points $\pm 1/d,$ $\pm {\pi(n+1/2)/l},$ $n \in \mathbb Z.$
Using the intermediate value theorem, one can show that each interval
\begin{equation*}I_n = \Big(\frac{\pi n}{l} - \frac{\pi}{2l}, \frac{\pi n}{l}+\frac{\pi}{2l}\Big), \quad n \in \mathbb N, \end{equation*}
contains at least one zero of $g(z).$ Let us choose as $z_n$ any zero from the interval $I_n,$ $n \in \mathbb N,$ and then all $z_n > 0$ are distinct.

Using the standard technique~\cite{[5]} involving application of Rouche's theorem, one can prove that the sequence of $c_2(z^2)$ zeros has the form $\{ -z_n \}_{n \ge 0} \cup \{ z_n\}_{n \ge 0}$ together with the asymptotic formulae
\begin{equation} \label{z_n}
z_n = \frac{\pi n}{l} + \frac{1}{d \pi n} + O\Big(\frac{1}{n^3}\Big), \quad  n \in \mathbb N.
\end{equation}
Let us prove that $\{ \sin z_n t\}_{n \ge 1}$ is a Riesz basis; for $\{ 1\} \cup \{ \cos z_n t\}_{n \ge 1}$ the proof is analogous.
According to Proposition~1.8.5 from~\cite{[5]}, it is sufficient to prove the completeness of the system $\{ \sin z_n t\}_{n \ge 1}.$

Let $f \in L_2(0, l)$ be such that $\int_0^{l} f(t) \sin z_n t \, dt = 0,$
then the function
\begin{equation*}
F(\lambda) = \frac{(\lambda - z^2_0) \int_0^{l} f(t) \sin \rho t \, dt}{\rho c_2(\lambda)},
\end{equation*}
is entire in $\lambda.$
Using the standard estimate (see~\cite {[5]})
\begin{equation*}|\sin \rho l| \ge M^{-1} e^{|\mathrm{Im}\, \rho|l}, \quad  |\rho| = \frac{\pi (n+\frac12)}{l}, \; n \in \mathbb{N}, \; n \ge N^*,\end{equation*}
we obtain that $|F(\lambda)| \le M$ on the circles $|\lambda| = (\pi (n+1/2)/l)^2$ for a sufficiently large $n\in \mathbb N.$
Then, the maximum modulus principle and Liouville's theorem yield that
$F(\lambda) \equiv C.$ By the Riemann--Lebesgue lemma, we also have
\begin{equation*}\int_0^{l} f(t) \sin \rho t \, dt = o(1), \quad  \frac{\rho c_2(\lambda)}{\lambda - z^2_0} = (-1)^{n+1} d^2(1+o(1)) \end{equation*}
for $\rho = \pi (n+1/2)/l.$
Hence, it follows that the only possible value is $C = 0.$ Now $f \equiv 0,$ and the completeness of the system $\{ \sin z_n t\}_{n \ge 1}$ is proved.
\end{proof}

We consider the following inverse problem.
\begin{ip} \label{inverse problem}
Given the spectrum $\{\lambda_n \}_{n \ge 0},$ recover the potential $q.$
\end{ip}
Now we prove a proposition that allows us to reduce the inverse problem to recovering $q$ from the characteristic function.
\begin{statement} \label{s1}
The characteristic function $\Delta(\lambda)$ is reconstructed uniquely by the spectrum $\{ \lambda_n \}_{n \ge 0}.$
\end{statement}
\begin{proof} Denote $s(\lambda) = \rho^{-1}\sin \rho \gamma.$
From representation~\eqref{Delta}, we easily obtain the asymptotics
\begin{equation} \label{Delta asymp}
\Delta(\lambda) = \left\{
\begin{array}{cc}
d^2 \rho \sin \rho l \cos \rho \gamma + O(e^{|\mathrm{Im}\, \rho|(\gamma+l)}), & \gamma \le l \text{ or } q(\gamma) = 0, \\[2mm]
-d q(\gamma) s^2(\lambda) + O(\rho e^{|\mathrm{Im}\, \rho|(\gamma+l)}), & \gamma > l \text{ and } q(\gamma) \ne 0,
\end{array} \quad \lambda \to \infty.
\right.
\end{equation}
Let $k_0$ be the multiplicity of $0$ in the spectrum. Without loss of generality, we assume that $0$ occurs only among the first $k_0$ spectrum members, i.e. $\lambda_0 = \ldots = \lambda_{k_0 - 2}=\lambda_{k_0-1}=0.$
By Hadamard's theorem, the characteristic function is determined up to the constant $C \ne 0:$
\begin{equation}
\Delta(\lambda) = C G(\lambda), \quad G(\lambda) := \lambda^{k_0} \prod^\infty_{n=k_0} \left(1-\frac{\lambda}{\lambda_n}\right).
\label{adamar}\end{equation}
According to~\eqref{Delta asymp}, the type of the constructed infinite product $G (\lambda)$ can be equal to either $ \gamma + l, $ or $ 2 \gamma.$ If the type is equal to $\gamma+l,$ then
the constant $C$ in~\eqref{adamar} is determined via the first asymptotic formula in~\eqref{Delta asymp}.
Now, let the type be $2\gamma,$ which is possible only for $\gamma > l$ and $q(\gamma) \ne 0.$ We define
\begin{equation*}
C_1 = \lim_{\lambda \to -\infty} \frac{G(\lambda)}{s^2(\lambda)}.
\end{equation*}
It is clear from~\eqref{Delta} that $C_1 = -d q(\gamma)/C.$
Also~\eqref{Delta} and~\eqref{adamar} yield that
\begin{equation*}G(\lambda)-C_1s^2(\lambda) = \frac{d^2 \rho \sin \rho d_2 \cos \rho \gamma (1+o(1))}{C}, \quad \lambda \to -\infty,
\end{equation*}
and
\begin{equation*}C = \lim_{\lambda \to -\infty} \frac{d^2 \rho \sin \rho l \cos \rho \gamma }{G(\lambda)-C_1s^2(\lambda)}.
\end{equation*}
Thus, $C$ is uniquely determined, which proves the assertion of the lemma.
\end{proof}

Let us prove a uniqueness theorem for solution of Inverse problem~\ref{inverse problem}. The proof is based on calculating the coefficients in the expansions of the function $q$ with respect to the basises $\{ \sin \pi n t/\gamma \}_{n \ge 1}$ and
$\{ \sin z_n t \}_{n \ge 1}$ on the segments $[0, \gamma]$ and $[a, b].$ These coefficients are calculated by substituting the values $\lambda = (\pi n/\gamma)^2$ and $\lambda=z_n^2$ into representation~\eqref{Delta} provided $c_2\big((\pi n/\gamma)^2\big)\ne 0,$ $n \in {\mathbb N}.$ This approach is similar to the one used in~\cite{wang-froz} for the uniqueness theorem proof.

\begin{thm} Denote by $\{ \tilde \lambda_n \}_{n \ge 0}$ the spectrum of boundary value problem~\eqref{*}--\eqref{Dirichlet} with some potential $\tilde q \in C(T).$
If the functions $c_2(\lambda)$ and $s(\lambda) := \rho^{-1}\sin{\rho \gamma}$  do not have common zeros, then the equality $\{ \lambda_n \}_{n \ge 0}=\{ \tilde \lambda_n \}_{n \ge 0}$ yields $q = \tilde q.$
\label{uniqueness}
\end{thm}

\begin{proof}
According to Proposition~\ref{s1}, the function $\Delta(\lambda)$ is uniquely determined by specifying its zeros $\{ \lambda_n \}_{n \ge 0}.$ We denote by $\tilde \Delta(\lambda)$ the characteristic function of boundary value problem~\eqref{*}--\eqref{Dirichlet} with the potential $\tilde q.$ So if
$\{ \lambda_n \}_{n \ge 0}=\{ \tilde \lambda_n \}_{n \ge 0},$ then $\Delta(\lambda)=\tilde\Delta(\lambda).$
Substituting the values $\lambda=(\pi n/\gamma)^2$ into representation~\eqref{Delta}, we obtain the relations
\begin{equation*}
\Delta\bigg(\Big(\frac{\pi n}{\gamma}\Big)^2\bigg) = k_n \bigg((-1)^{n+1} - \frac{\gamma}{\pi n} \varkappa_n\bigg), \quad k_n := c_2\bigg(\Big(\frac{\pi n}{\gamma}\Big)^2\bigg), \; \varkappa_n := \int_0^{\gamma} \sin \frac{\pi n}{\gamma}t \, q(t) \, dt,
\end{equation*}
whence we arrive at the formula
\begin{equation} \label{varkappa}
\varkappa_n = -\frac{\pi n}{k_n \gamma} \left( \Delta\bigg(\Big(\frac{\pi n}{\gamma}\Big)^2\bigg) + (-1)^n k_n \right), \quad n \ge 1.
\end{equation}
Analogously, we get the equality
\begin{equation*}
\int_0^{\gamma} \sin \frac{\pi n}{\gamma}t \, \tilde q(t) \, dt = \varkappa_n, \quad n \ge 1,
\end{equation*}
with the same $\varkappa_n.$ Now
\begin{equation*}
\int_0^{\gamma} \sin \frac{\pi n}{\gamma}t \, (q(t) - \tilde q(t)) \, dt = 0, \quad n \ge 1,
\end{equation*}
and the completeness of the system $\{ \sin{\pi n t/\gamma}\}_{n \ge 1}$ in  $L_2(0, \gamma)$ yields that $q = \tilde q$ on $[0, \gamma].$

Further, we substitute $\lambda = z_n^2$ into the characteristic function. The completeness of the system $\{ \sin z_n t\}_{n \ge 1}$ in $L_2(0, l)$ is proved in Lemma~\ref{riesz}. Proceeding analogously to the first part of the proof, for the coefficients $\xi_n := \int_0^{l} \sin z_nt \, q(b-t)\, dt,$ we obtain the formulae
\begin{equation} \label{xi}
\xi_n =  -\frac{z_n^2}{\sin z_n \gamma}\left(\Delta\big(z^2_n\big) + c_1\big(z^2_n\big) \frac{\sin z_n \gamma}{z_n} +d q(\gamma)\frac{\sin^2 z_n \gamma}{z_n^2}\right), \quad n \ge 1,
\end{equation}
and the equality $q = \tilde q$ on $[a, b].$
\end{proof}
Based on formulae~\eqref{varkappa} and~\eqref{xi}, we obtain the following algorithm of recovery.

\begin{alg} \label{algorithm}
Let the functions $c_2(\lambda)$ and $s(\lambda)$ not have common zeros. To recover the potential from $\{ \lambda_n \}_{n \ge 0},$ one should:

1. Construct $\Delta(\lambda),$ see Proposition~\ref{s1}.

2. Compute the coefficients $\varkappa_n$ via formula~\eqref{varkappa} and find \begin{equation*}
q(t) = \frac{2}{\gamma} \sum_{n=1}^\infty \varkappa_n \sin \frac{\pi n}{\gamma}t, \quad t \in (0, \gamma).\end{equation*}

3. Compute the coefficients $\xi_n$ via formula~\eqref{xi} and construct
\begin{equation*}q(b - t) = \sum_{n=1}^\infty \xi_n \chi_n(t), \quad t \in (0, l),\end{equation*}
where the system $\{ \chi_n(t)\}_{n \ge 1}$ is biorthogonal to the basis $\{\sin z_n t\}_{n \ge 1}$ in $L_2(0, l).$
\end{alg}

\begin{statement} It can be ensured that $c_2(\lambda)$ and $s(\lambda)$ do not have common zeros by imposing one of the three restrictions on $d,$ $\gamma,$ and $l:$

1) for some $k \in \mathbb N,$ $l = k \gamma$ holds;

2) the numbers $\pi l/\gamma$ and $\pi d/\gamma$ are rational;

3)  the numbers $l/\gamma,$ $\pi d/\gamma$ are rational, and $\cos l/d \ne 0.$
\end{statement}
\begin{proof}
If 1) holds, then $c_2\big((\pi n/\gamma)^2\big)=(-1)^{nk} d.$
Suppose 2) or 3) holds, and $c_2\big((\pi n/\gamma)^2\big)=0$ for some $n \in \mathbb N.$
Then $\cos \pi n l/\gamma \ne 0,$ and we have
\begin{equation} \bigg(1 - \Big(\frac{\pi n}{\gamma}d\Big)^2\bigg) \tan \frac{\pi n l}{\gamma} = \frac{\pi n}{\gamma}d.
\label{irrat}\end{equation}

If 2) holds, then $\pi n l/\gamma \in \mathbb Q,$ and $\tan \pi n l/\gamma$ is an  irrational number, see~\cite{niven}. Therefore, in the left part of~ \eqref{irrat} we have either 0 or an irrational number, and in the right part a non-zero rational number, which leads to a contradiction.

If 3) holds, then $\pi n l/\gamma$ denotes a rational number of degrees, and, according to~\cite{niven}, the inclusion
\begin{equation*}\tan \frac{\pi n l}{\gamma} \in (\mathbb R \setminus \mathbb Q) \cup \{ 0\} \cup \{ \pm 1 \}\end{equation*}
is fulfilled.
The cases when $\tan \pi n l/\gamma$ is irrational or zero are analogous to 2). If $\tan \pi n l/\gamma=\pm1,$ then~\eqref{irrat} also leads to a contradiction because the rational number $\pi n d/\gamma$ can not satisfy any of the quadratic equations $1 - x^2 = \pm x.$
\end{proof}
\section{Necessary and sufficient conditions}
Further, we obtain necessary and sufficient conditions on the spectrum in the case
\begin{equation} \label{case}
l=\gamma, \quad q \in W^1_2[0, \gamma] \cap W^1_2[a, b].\end{equation}
Provided these restrictions are satisfied, formula~\eqref{Delta} can be rewritten in the form
\begin{equation} \begin{split}
\Delta(\lambda) = \frac{d^2 \rho}{2} \sin 2 \rho l - d \cos 2 \rho l - \frac{\sin 2 \rho l}{\rho} -d q(\gamma) \frac{\sin^2 \rho l}{\rho^2} +  \\
\left[ \frac{\rho^2 d^2 - 1}{\rho^2} \sin \rho l -  \frac{d}{\rho} \cos \rho l\right] \int_0^l q(t) \sin \rho t \, dt
- \frac{\sin \rho l}{\rho^2} \int_0^l \sin \rho t q(b-t) \, dt.
\end{split}
\label{2}
\end{equation}
Substituting  $\lambda = (\pi n/l)^2$ and $\lambda = z_n^2$ into~\eqref{2}, using integration by parts, we get the following proposition.
\begin{statement} \label{necessity con}
For $n \in \mathbb N,$ the formulae
\begin{equation*}\Delta\bigg(\Big(\frac{\pi n}{l}\Big)^2\bigg) = -d + d \bigg(\frac{l}{\pi n}\bigg)^2 \big[q(l) - (-1)^n q(0)+  \kappa_n\big], \end{equation*}
\begin{equation*}\Delta(z_n^2) =
(-1)^n \frac{\sin z_n l}{z_n^3}  \Big[ \frac{1}{d^2} + q(a) - q(l) - (-1)^n q(b) +\eta_n \Big]\end{equation*}
 are fulfilled with some $\{ \kappa_n \}_{n \ge 1}, \{ \eta_n \}_{n \ge 1} \in l_2.$
\end{statement}

Denote
\begin{equation*}Q_1(z) = \int_z^{b} q(t) \, dt, \quad Q_2(z) = \int_0^{l-z} q(a+t) \, dt.\end{equation*}
Integrating by parts, we obtain the following representation:
\begin{equation} \begin{split}
\Delta(\lambda) =  \frac{d^2 \rho}{2} \sin 2 \rho l - d \cos 2 \rho l - \frac{\sin 2 \rho l}{\rho} + \\
+d^2 q(0) \frac{\sin \rho l}{\rho} - d^2 \frac{\sin 2 \rho l}{2 \rho} q(l) + \frac{1}{2\rho} \int_0^{2l} \sin \rho t\, W(t) \, dt,
\end{split}
\label{1}
\end{equation}
where the function $W \in L_2[0, 2l]$ has the form
\begin{equation*}W(t) = \left\{
\begin{array}{cc}
 d q(l-t) - dq(l) - Q_1(l-t) + d^2 q'(l-t) - Q_2(t), & t \in [0, l) ; \\
-d q(t-l) - dq(l) - Q_1(t-l) + d^2 q'(t-l) - Q_2(2l-t), & t \in [l, 2l].
\end{array}
\right. \end{equation*}

With the help of formula~\eqref{1}, by the standard technique~\cite{[5]}, asymptotics of the spectrum can be obtained, i.e. the following theorem holds.
\begin{thm}
The following asymptotic formulae hold:
\begin{equation} \label{asymp}
\lambda_n = \rho_n^2, \quad \rho_n = \frac{\pi n}{2 l} + \frac{2}{d \pi n} + \frac{4 l \delta_n q(0)}{\pi^2 n^2} + \frac{\mu_n}{n^2}, \quad n \in {\mathbb N},
\end{equation}
where $\delta_n = \sin \pi n/2,$ $\{ \mu_n\}_{n \ge 1} \in l_2.$
\end{thm}
Using this asymptotics, we refine formula~\eqref{adamar} in the following proposition.
\begin{statement}
The characteristic function is reconstructed by the formula
\begin{equation} \label{prod}
\Delta(\lambda) = l d^2 (\lambda - \lambda_0)\prod_{n=1}^\infty \frac{\lambda_n - \lambda}{(\frac{\pi n}{2l})^2}
\end{equation}
\end{statement}
\begin{proof}
Introduce the function
\begin{equation*}\Delta^*(\lambda) = \frac{\rho d^2}{2} \sin 2 \rho l =\frac{\lambda d^2}{4l} \prod_{n=1}^\infty \left(1 - \frac{\lambda}{(\frac{\pi n}{2 l})^2}\right).\end{equation*}
Let the assumptions about the spectrum made in Proposition~\ref{s1} hold.
Using~\eqref{adamar}, for $k_0 \ge 1,$ we have
\begin{equation*}\frac{\Delta(\lambda)}{\Delta^*(\lambda)} = C\frac{4l}{d^2}  \prod_{n=0}^{k_0-1} \frac{\lambda (\frac{\pi n}{2l})^2}{ (\frac{\pi n}{2l})^2 - \lambda} \prod_{n=k_0}^\infty \frac{ (\frac{\pi n}{2l})^2}{\lambda_n} \frac{\lambda_n - \lambda}{ (\frac{\pi n}{2l})^2 - \lambda}.
\end{equation*}

Let us prove that
\begin{equation} \label{lim1}
\lim_{\lambda \to - \infty} \prod_{n=k_0}^\infty \frac{\lambda_n - \lambda}{ (\frac{\pi n}{2l})^2 - \lambda} = 1,
\end{equation}
which is equivalent to
 $$\lim_{\lambda\to-\infty}\sum_{n=k_0}^\infty \ln (1 + x_n(\lambda)) = 0, \quad x_n(\lambda) := \frac{\lambda_n -  (\frac{\pi n}{2l})^2}{ (\frac{\pi n}{2l})^2 - \lambda}.$$
Since $\lambda_n -  [\pi n/(2l)]^2=O(1),$ for a sufficiently small $\lambda<0,$ we have $|x_n(\lambda)| <1/2,$ and using the Taylor series for $\ln (1+ x_n(\lambda)),$ we obtain the inequality
 $$\left|\sum_{n=k_0}^\infty \ln (1 + x_n(\lambda))\right| \le \sum_{n=k_0}^\infty |\ln (1 + x_n(\lambda)| \le \sum_{n=k_0}^\infty 2|x_n(\lambda)|.$$
In turn, one can estimate
 $$|x_n(\lambda)| \le \frac{M}{(\frac{\pi n}{2l})^2+|\lambda|} \le \frac{M}{\frac34 (\frac{\pi n}{2l})^2+\frac14 |\lambda|}.$$
By Young's inequality, we get $|x_n(\lambda)|\le M[\pi n/(2l)]^{-3/2}|\lambda|^{-1/4}$ and
 $$\sum_{n=k_0}^\infty |x_n(\lambda)| \le \frac{M_2}{|\lambda|^{\frac14}} \to 0, \quad \lambda \to -\infty.$$
 Using three previous formulae, we obtain~\eqref{lim1}.

From~\eqref{1} it is clear that $\Delta(\lambda)/\Delta^*(\lambda) \to 1$ as $\lambda \to -\infty.$ Taking into account~\eqref{lim1}, we calculate
$$C=  \frac{d^2}{4l} (-1)^{k_0-1} \prod_{n=1}^{k_0-1} \Big(\frac{2 l}{\pi n}\Big)^{2} \prod_{n=k_0}^\infty \frac{\lambda_n}{(\frac{\pi n}{2 l})^2}.$$
Substituting this formula into \eqref{adamar}, we obtain~\eqref{prod}.

The case  $k_0=0$ is proceeded analogously.
\end{proof}

Now, let an arbitrary sequence $\{ \lambda_n\}_{n\ge 0}$ be given the terms of which satisfy the asymptotic formulae
 \begin{equation} \label{sufficiency asymp}
\lambda_n = \rho_n^2, \quad \rho_n = \frac{\pi n}{2l} + \frac{2}{d \pi n} + \frac{4 l \delta_n u}{\pi^2 n^2} + \frac{\mu_n}{n^2}, \quad n \in \mathbb N,
\end{equation}
where $u \in \mathbb C,$ $\{ \mu_n\}_{n \ge 1} \in l_2.$
Our next goal is to obtain the conditions which one should impose on this sequence to ensure it is the spectrum of some boundary value problem~\eqref{*}--\eqref{Dirichlet} in particular case~\eqref{case}.
\begin{thm} Let the function $\Delta(\lambda)$ be constructed via formula \eqref{prod} with arbitrary numbers $\lambda_n$ of form \eqref{sufficiency asymp}, then the following representation holds:
\begin{equation} \label{sufficiency Delta}
\Delta(\lambda) = \tilde{\Delta}(\lambda) + C_0 \frac{\sin 2 \rho l}{\rho} + \int_0^{2l} W(t) \frac{\sin \rho t}{2\rho} dt, \quad W \in L_2(0, 2l),
\end{equation}
where
\begin{equation*}\tilde{\Delta}(\lambda) = \frac{d^2 \rho}{2} \sin 2 \rho l - d \cos 2 \rho l + d^2 u \frac{\sin \rho l}{\rho}.\end{equation*}
\end{thm}
\begin{proof} The proof of the theorem is based on the approach used in~\cite{results}, where the representation for another characteristic function was obtained with fewer terms.

1. Denote by $\tilde{\lambda}_n = \tilde \rho_n^2,$ $n \ge 0,$ the zeros of the function $\tilde{\Delta}(\lambda),$ wherein $\tilde \rho_n$ have the same asymptotics as $\rho_n$ in formula \eqref{asymp}.
Note that $\tilde \Delta(\lambda)$ is recovered by formula~\eqref{prod}  provided $ \lambda_n $ is replaced  by $\tilde \lambda_n$ in it.

We consider the numbers
\begin{equation*} \theta_k =\frac{\pi k}{2l} \big[\Delta(\lambda) - \tilde{\Delta}(\lambda)\big]\Big|_{\lambda=(\frac{\pi k}{2l})^2}, \quad k \in {\mathbb N}.\end{equation*}
Introduce the function $\Delta^*(\lambda) =d^2 \rho/2 \sin 2 \rho l.$ Using~\eqref{prod}, we write
\begin{equation*}\Delta(\lambda) = \Delta^*(\lambda) F(\lambda), \quad F(\lambda) := \prod_{n=0}^\infty \frac{\lambda_n - \lambda}{(\frac{\pi n}{2l})^2-\lambda}, \end{equation*}
and, analogously,
\begin{equation*}\tilde\Delta(\lambda) = \Delta^*(\lambda) \tilde F(\lambda), \quad \tilde F(\lambda) := \prod_{n=0}^\infty \frac{\tilde\lambda_n - \lambda}{(\frac{\pi n}{2l})^2-\lambda}.\end{equation*}
Taking into account these represenations, we have
\begin{equation} \theta_k = \frac{\pi k}{2l} \left.[\Delta^*(\lambda)]'\right|_{\lambda=(\frac{\pi k}{2l})^2} \left[\Big(\tilde\lambda_k - \Big(\frac{\pi k}{2l}\Big)^2\Big) \tilde d_k - \Big(\lambda_k - \Big(\frac{\pi k}{2l}\Big)^2\Big) d_k \right], \quad k \in {\mathbb N},
\label{theta_k}\end{equation}
where
\begin{equation*} d_k = \prod_{n \ne k} \frac{\lambda_n - (\frac{\pi k}{2l})^2}{(\frac{\pi n}{2l})^2 - (\frac{\pi k}{2l})^2}, \quad \tilde d_k = \prod_{n \ne k} \frac{\tilde \lambda_n - (\frac{\pi k}{2l})^2}{(\frac{\pi n}{2l})^2 - (\frac{\pi k}{2l})^2}.
\label{d_k}\end{equation*}
Further, we estimate $d_k$ (for $\tilde d_k,$ the calculations are analogous) and $\tilde d_k - d_k.$ We rewrite the first coefficient in the form
\begin{equation} d_k = \prod_{\substack{n\le N,\\ n \ne k}} \big(1+ x_{n,k}\big) \exp H_k, \quad H_k := \sum_{\substack{n > N, \\ n \ne k}} \ln \big(1 + x_{n,k}\big), \quad x_{n,k} := \frac{\lambda_n - (\frac{\pi n}{2l})^2}{(\frac{\pi n}{2l})^2-(\frac{\pi k}{2l})^2}, \label{star}
\end{equation}
with a certain fixed $N.$
By virtue of~\eqref{sufficiency asymp}, we have $\lambda_n - [\pi n/(2l)]^2 = O(1),$ and one can choose $N$ independent of $k$ such that $|x_{n,k}| < 1/2$ for $n > N.$ Using the Taylor series for $\ln (1 + x_{n,k}),$ we can get the estimate
\begin{equation}|H_k| \le 2 \sum_{\substack{n > N, \\ n \ne k}} |x_{n,k}| \le 2C \sum_{n > k} \frac{1}{(n-k)^2} = O(1).
\label{H_k}\end{equation}
Then, from~\eqref{star} we obtain
$d_k = O(1).$  Let us prove that $d_k-\tilde d_k=O(1/k^2).$ For this purpose, we write
\begin{equation*}d_k - \tilde d_k = \Biggr( \prod_{\substack{n \le N, \\ n \ne k}} \big(1 + x_{n,k}\big) - \prod_{\substack{n\le N, \\ n \ne k}} \big(1 + \tilde x_{n,k}\big)\Biggr) \exp H_k +
 \prod_{\substack{n\le N, \\ n \ne k}} \big(1 + \tilde x_{n,k}\big) \big(\exp H_k - \exp \tilde H_k\big),\end{equation*}
 where $\tilde H_k$ and $\tilde x_{n,k}$ are introduced analogously to $H_k$ and $x_{n,k}.$
The first summand in this equality is estimated as $O(1/k^2)$ since $x_{n,k} = O(1/{k^2})$ and $\tilde x_{n,k} = O(1/{k^2})$ for $n<N,$ and~\eqref{H_k} holds as well.
To estimate the second summand, we note that from the Taylor series for $\ln (1 + x_{n,k}),$ $k > N,$ one can obtain the formula
\begin{equation*}H_k=\sum_{\substack{n>N, \\ n\ne k}} \Big(x_{n,k} +O\big(x^2_{n,k}\big)\Big);\end{equation*}
the analogous formula holds for $\tilde H_k.$ Using them and the estimate for the difference between the values of a function in terms of its derivative, we obtain
\begin{multline*} \big|\exp H_k - \exp \tilde H_k\big| \le \max_{z \in [H_k, \tilde H_k]} |\exp z| \,\big|H_k - \tilde H_k \big| \le \\
\le C \sum_{\substack{n>N, \\ n\ne k}} \big|x_{n,k} - \tilde x_{n,k}\big| + C \sum_{\substack{n>N, \\ n\ne k}} \frac{1}{k^2(n-k)^2} \le C\sum_{\substack{n>N, \\ n\ne k}} \frac{\hat\mu_n}{n(n^2 - k^2)}+\frac{C}{k^2},\end{multline*}
where $\{ \hat\mu_n \}_{n \ge 0} \in l_2$ by virtue of asymptotic formulae~\eqref{sufficiency asymp}.
Moreover, we have
\begin{equation*}\sum_{\substack{n\ge 1, \\ n\ne k}} \frac{\hat\mu_n}{n(n^2 - k^2)} \le \frac{1}{k} \sum_{\substack{n\ge 1, \\ n\ne k}} \frac{\hat\mu_n}{n(n-k)}\le \frac{1}{k} \sqrt{\sum_{n=1}^\infty \hat\mu_n^2} \sqrt{\sum_{\substack{n\ge 1, \\ n\ne k}} \frac{1}{n^2 (n-k)^2}} \le \frac{C}{k^2},\end{equation*}
which implies $|\exp H_k - \exp \tilde H_k|=O(1/k^2)$ and $d_k-\tilde d_k=O(1/k^2).$

It follows from asymptotics~\eqref{sufficiency asymp} that
\begin{equation*}\lambda_k - \tilde \lambda_k = \frac{\nu_k}{k}, \quad \lambda_k - \Big(\frac{\pi k}{2l}\Big)^2 = O(1), \; \tilde\lambda_k - \Big(\frac{\pi k}{2l}\Big)^2 = O(1),\end{equation*}
where $\{ \nu_k \}_{k \ge 1} \in l_2.$
Now
\begin{equation*}k \left[ \bigg(\lambda_k - \Big(\frac{\pi k}{2l}\Big)^2\bigg) d_k - \bigg(\tilde\lambda_k - \Big(\frac{\pi k}{2l}\Big)^2\bigg) \tilde d_k \right] = k \left[ d_k(\lambda_k - \tilde \lambda_k) + \bigg(\tilde\lambda_k - \Big(\frac{\pi k}{2l}\Big)^2\bigg)(d_k - \tilde d_k) \right] \in l_2,\end{equation*}
and from~\eqref{theta_k} we get $\{ \theta_k\}_{k \ge 1} \in l_2. $

2. Since the system  $\{ \sin \pi kt/(2l) \}_{k \ge 1}$ is an orthogonal basis in the space $L_2(0, 2l),$ there exists a function $W \in L_2(0, 2l)$ such that
\begin{equation*}\int_0^{2l}  W(t) \sin \frac{\pi k}{2l} t\, dt = 2 \theta_k.\end{equation*}
Introduce
\begin{equation*}\theta(\rho) = \int_0^{2l} W(t) \sin \rho t \, dt, \quad P(\lambda) = \lambda \frac{\theta(\rho)/(2\rho) - \Delta(\lambda) + \tilde\Delta(\lambda)}{\Delta^*(\lambda)} = \frac{\theta(\rho)}{d^2\sin 2\rho l} + \lambda \big(\tilde F(\lambda) - F(\lambda)\big),\end{equation*}
wherein the function $P(\lambda)$ is entire. Proceeding analogously to the proof in part~1, one can show that $|\tilde F(\lambda) - F(\lambda)| \le M/k^{2}$
on each circle $|\lambda| = [\pi k/(2l) + \pi/(4l)]^2,$ $k>N.$
Hence, it follows that the function $P(\lambda)$ is bounded by modulus in the whole complex plane, and, by Liouville's theorem, $P(\lambda) \equiv C_0.$ Thus, representation~\eqref{sufficiency Delta} is proved.
\end{proof}

Let us construct $\Delta(\lambda)$ by formula~\eqref{prod} with given $\lambda_n$ having asymptotics~\eqref{sufficiency asymp}. It is proved that this function has  form~\eqref{sufficiency Delta} with some $W \in L_2(0, 2l)$ and $C_0 \in \mathbb C.$ It is easy to see that the function $W$ is uniquely determined:
\begin{equation} \label{W recovery}
W(t) = \frac1l \sum_{n=1}^\infty \beta_n \sin \frac{\pi n}{2l}t, \quad
\beta_n := \frac{\pi n}{l} \Delta\bigg(\Big(\frac{\pi n}{2l}\Big)^2 \bigg) + \frac{ \pi n}{l}(-1)^nd  - 2 d^2 u \delta_n, \; n \in \mathbb N.
\end{equation}
According to formula~\eqref{1}, the constructed $\Delta(\lambda)$ is the characteristic function of the boundary value problem only if, for $W,$ there is the representation
\begin{equation}W(t) = \left\{
\begin{array}{cc}
 d g(l-t) - dg(l) - G_1(l-t) + d^2 g'(l-t) - G_2(t), & t \in [0, l), \\
-d g(t-l) - dg(l) - G_1(t-l)+ d^2 g'(t-l) - G_2(2l-t), & t \in [l, 2l],
\end{array}
\right. \label{W}\end{equation}
with functions $g,$ $G_1,$ and $G_2$ satisfying the following conditions:
\begin{enumerate}

\item[(A)]  $g \in W^1_2[0, l],$ $G_2 \in W^2_2[0, l],$ and $G_1(z) = \int_z^l g(t) \, dt,$ $z \in [0, l];$

\item[(B)] $g(l) = -2(C_0+1)/d^2$ and $g(0) = u;$

\item[(C)] $ G_2(l) = 0.$
\end{enumerate}
It is easy to see that if representation~\eqref{W} holds,
one can reconstruct $g,$ $G_1,$ and $G_2$ sequentially by the formulae
\begin{equation} \left.
\begin{array}{c}
\displaystyle g(t) = \frac{1}{2d} \{ W(l-t) - W(l+t) \}, \\[2mm]
G_1(t) = \int_t^l g(w) \, dw, \\[3mm]
G_2(t) =
-d g(l-t) - dg(l) - G_1(l-t) + d^2 g'(l-t) - W(2l - t),
\end{array}\right\}\label{gg}\end{equation}
where $t \in [0, l].$

Conversely, suppose the function $W(t)$ from representation~\eqref{sufficiency Delta} has form~\eqref{W}. We consider
\begin{equation*}q(t)  = \left\{
\begin{array}{cc}
g(t), & t \in [0, \gamma],\\
G'_2(l-t+a), & t \in [a, b].
\end{array}\right.\end{equation*}
Condition (A) yields that $q \in W^1_2[0, \gamma] \cap W^1_2[a, b].$
Comparing formula~\eqref{1} with \eqref{sufficiency Delta} under conditions (B)--(C), we conclude that the characteristic function of boundary value problem~\eqref{*}--\eqref{Dirichlet} with the constructed $q$ coincides with $\Delta(\lambda),$ and $\{ \lambda_n \}_{n \ge 0}$ is the spectrum of this boundary value problem.

Further, we obtain conditions that imply form~\eqref{W} of the function $W.$ For this purpose, we need two following lemmas.
\begin{lemma}[\cite{results}] Suppose $\{ \gamma_k\}_{k \ge 1} \in l_2$ and  $f \in L_2(0, l).$ We have $f \in W^1_2[0, l]$ if and only if the asymptotics
\begin{equation*}\int_0^l  f(x) \sin \Big(\frac{\pi k}{l} + \gamma_k\Big)x \, dx = \frac{l}{\pi k}\big( w_1 - (-1)^k  w_2\big) + \frac{\tilde \gamma_k}{k}, \; k \in {\mathbb N}, \quad \{ \tilde \gamma_k \}_{k \ge 1} \in l_2, \end{equation*}
holds, wherein $w_1=f(0),$ $w_2=f(l).$
    \label{lemma buterin}
 \end{lemma}

  \begin{lemma}
 Suppose $f \in L_2(0, l).$ We have $f \in W^2_2[0, l]$ if and only if
  \begin{equation}\int_0^l  f(x) \cos z_k x \, dx = \frac{1}{z^2_k} \Big( (-1)^k v_2 - v_1 + \tilde\gamma_k \Big), \;  k \in {\mathbb N}, \quad \{\tilde \gamma_k \}_{k \ge 1} \in l_2.
  \label{l2 statement}
  \end{equation}
 Moreover, $v_1 = f'(0)$ and $v_2 = f'(l) + f(l)/d.$
\label{my lemma}
  \end{lemma}
  \begin{proof}
  Necessity. If $f \in W^2_2[0, l],$ then we obtain formula~\eqref{l2 statement} with $v_1 = f'(0)$ and $v_2 = f'(l) + f(l)/d$ integrating by parts two times.

Sufficiency. Let formula~\eqref{l2 statement} hold.
Let us construct the function $g \in W^2_2[0, l]$ such that its coefficients with respect to the system $\{ 1 \} \cup \{ \cos z_n t \}_{n \ge 1}$ coincide with the coefficients of $f.$
We find $g$ in the form $g(x) = \int_x^l \tilde g(t) \, dt + C.$ Calculating the coefficient with respect to $\cos z_k t,$ using integration by parts, we obtain
  \begin{equation*}
  \frac{C \sin z_k l}{z_k} + \frac{1}{z_k}\int_0^l \tilde g(t) \sin z_k t \, dt = \frac{1}{z^2_k} \Big( (-1)^k v_2 - v_1 + \tilde\gamma_k \Big).
  \end{equation*}
   Therefore, $\tilde g(x) = h(x)-C F_s(x),$  where the functions $h$ and $F_s$ are uniquely determined by the equalities
    \begin{equation}
    \int_0^l h(t) \sin z_k t \, dt = \frac{1}{z_k} \Big( (-1)^k v_2 - v_1 + \tilde\gamma_k \Big), \quad  \int_0^l F_s(t) \sin z_k t \, dt = \sin z_k l, \quad k \ge 1.
      \label{lemma def}
    \end{equation}
    Formula~\eqref{z_n} yields that $\{ k(\sin z_k l-(-1)^k l/(d\pi k)) \}_{k \ge 1} \in l_2.$ By Lemma~\ref{lemma buterin}, the inclusions $h, F_s \in W^1_2[0, l]$ hold. To construct the required $g \in W^2_2[0, l],$ it remains to choose $C$ such that
\begin{equation*}
\int_0^l f(t) \, dt = \int_0^l g(t) \, dt = \int_0^l \int_x^b h(t) \, dt - C \int_0^l t F_s(t) \, dt + Cl.
\end{equation*}
Such constant $C$ exists if $\int_0^l t F_s(t) \, dt \ne l.$

Let us prove that $\int_0^l t F_s(t) \, dt \ne l.$ Assume the converse. Construct the entire function
\begin{equation*}
F(\lambda) = \frac{(\lambda - z_0^2) (\sin \rho l - \int_0^l F_s(t) \sin \rho t \, dt)}{\rho c_2(\lambda)},
\end{equation*}
where $\{ z_n^2 \}_{n \ge 0}$ denotes the sequence of $c_2(\lambda)$ zeros.  Proceeding analogously to the proof of Lemma~\ref{riesz}, we get $F(\lambda) \equiv C,$ which infers $\sin \rho l -\int_0^l F_s(t) \sin \rho t \, dt \equiv C \rho c_2(\lambda)/(\lambda - z_0^2).$ Substituting $\rho = \pi n/l + \pi/(2l)$ into the previouos equality, using the Riemann--Lebesgue lemma, we compute $C = -1/d^2.$ Hence
\begin{equation}
\int_0^l F_s(t) \sin \rho t \, dt \equiv \frac{1}{d^2(\lambda - z_0^2)} \big[(1-z_0^2) \sin \rho l + d \rho \cos \rho l \big].
\label{equiv2}
\end{equation}
It is clear that $F_s(t)$ is uniquely determined by the coefficients
\begin{equation*}
\int_0^l F_s(t) \sin \frac{\pi n}{l}t \, dt =  \frac{\pi n}{d l} \frac{(-1)^n}{(\frac{\pi n}{l})^2 -  z^2_0}, \; z_0 \ne \frac{\pi n}{l}, \quad n \in \mathbb N.
\end{equation*}
By virtue of the system $\{ \sin \pi n t/l \}_{n \ge 1}$ completeness, we have $F_s(t) = -\sin z_0 t /(d \sin z_0 l).$ Substituting this function into~\eqref{equiv2}, we arrive at a contradiction. Therefore, $\int_0^l t F_s(t) \, dt \ne l.$

We have proved that there exists a function $g\in W^2_2[0, l]$ having the same coefficients with respect to the system $\{ 1 \} \cup \{ \cos z_n t \}_{n \ge 1}$ as $f$ does.
Since this system is complete, we obtain $f = g.$ Comparing~\eqref{l2 statement} with the formula written by necessity, we arrive at $v_1 = f'(0)$ and $v_2 = f'(l) + f(l)/d.$
  \end{proof}

\begin{thm} For the function $W(t)$ from \eqref{sufficiency Delta} to have the form~\eqref{W} under conditions (A)--(B),
the following asymptotic formulae are sufficient:
\begin{equation}\Delta\bigg(\Big(\frac{\pi n}{l}\Big)^2\bigg) = -d + d \bigg(\frac{l}{\pi n}\bigg)^2 \big[c - (-1)^n u +  \kappa_n\big], \quad c := -\frac{2(C_0+1)}{d^2},
\label{con1}\end{equation}
\begin{equation}\Delta(z_n^2) =
 \frac{\sin z_n l}{z_n^3} \Big[ h_2 - (-1)^n h_1 + \eta_n \Big]
 \label{con2}\end{equation}
with some $h_1, h_2 \in \mathbb{C}$ and sequences $\{ \kappa_n \}_{n \ge 1}, \{ \eta_n \}_{n \ge 1} \in l_2.$
\end{thm}
\begin{proof} Let us construct the functions $g$ and $G_2$ via formulae~\eqref{gg}. It is sufficient to prove that the inclusions $g \in W^1_2[0, l]$ and $G_2 \in W^2_2[0, l]$ hold, and besides $g(0) = u,$ $g(l) = c.$

  Substituting $\rho=\pi n/l$ into \eqref{sufficiency Delta}, we obtain
\begin{equation*}\Delta\bigg(\Big(\frac{\pi n}{l}\Big)^2\bigg)=-d + \frac{l}{2 \pi n} \int_0^{2l} W(t) \sin \frac{\pi n}{l}t \, dt.
\end{equation*}
Splitting the previous integral into two, we have
 \begin{equation*}\int_0^{l} W(l-t) \sin \frac{\pi n}{l}(l-t)\,dt+ \int_0^{l} W(l+t) \sin \frac{\pi n}{l}(l+t) \, dt= (-1)^{n+1} \int_0^{l} 2d g(t) \sin \frac{\pi n}{l}t \, dt.\end{equation*}
Therefore,~\eqref{con2} yields the asymptotic
\begin{equation*}\int_0^l g(t) \sin \frac{\pi n}{l}t \, dt = \frac{l}{\pi n}[u - (-1)^n c - (-1)^n \kappa_n].\end{equation*}
  By Lemma~\ref{lemma buterin}, $g \in W^1_2[0, l]$ and $g(0) = u,$ $g(l) = c.$

Substituting $\rho=z_n$ into~\eqref{sufficiency Delta}, we have
\begin{equation*}\Delta(z_n^2) = \frac{(C_0+1) \sin 2z_n l}{z_n} + d^2 u \frac{\sin z_n l}{z_n} + \frac{\sin^2 z_n l}{dz_n^2} + \frac{1}{2z_n} \int_0^{2l} W(t) \sin z_n t\, dt.
\end{equation*}
We transform the integral as follows:
\begin{multline*}\frac12\int_0^{2l} W(t) \sin z_n t\, dt =\frac12\int_0^{l} W(t) \sin z_n t\, dt + \frac12\int_0^{l} W(2l - t) \sin z_n(2l - t) \, dt \\
=\int_0^l d g(l-t) \sin z_n t \, dt + \sin z_n l \int_0^l W(2l - t)  \cos z_n(l-t) \, dt \\
= \sin z_n l\int_0^l W(l+t) \cos z_n t\, dt + d \int_0^l g(t) \sin z_n(l-t) \, dt \\
=\sin z_n l\int_0^l W(l+t) \cos z_n t \, dt + d \sin z_n l \int_0^l g(t) \cos z_nt \, dt - d \cos z_n l \int_0^l g(t) \sin z_nt \, dt.\end{multline*}
Making the substitution $\cos z_n l = [d z_n - 1/(d z_n)] \sin z_n l$ in the last term and integrating by parts, we obtain
\begin{equation*} \Big(d^2 z_n - \frac{1}{z_n}\Big) \int_0^l g(t) \sin z_n t \, dt = d^2\bigg[u - c \cos z_n l  + \int_0^l g'(t) \cos z_n t \, dt\bigg] + \int_0^l G_1(t) \cos z_n t\, dt.\end{equation*}
Finally, taking into account~\eqref{gg} and $c = -2(C_0+1)/d^2,$ we have
\begin{equation*}
\Delta(z_n^2) = \frac{\sin^2 z_n l}{dz_n^2} + \frac{\sin z_n l}{z_n} \int_0^{l} f(t) \cos z_n t\, dt, \quad f(t) = -G_2(l-t) - 2 G_1(t) - d g(l).
\end{equation*}
Equating the obtained expression to the right side in~\eqref{con2} and dividing by $z_n^{-1}\sin z_nl,$ we arrive at the formula
\begin{equation*}
\int_0^l f(t) \cos z_n t \, dt =  \frac{1}{z_n^2} \Big[ h_2 - (-1)^n h_1 + \eta_n \Big] - \frac{\sin z_n l}{z_n d} = \frac{1}{z_n^2} \Big[ h_2 - (-1)^n \Big(h_1+\frac{1}{d^2}\Big) + \tilde\eta_n \Big], \; n \ge 1,
\end{equation*}
 where $\{ \tilde \eta_n\}_{n \ge 1} \in l_2.$
Applying Lemma~\ref{my lemma}, we obtain $f \in W^2_2[0, l]$ and $G_2 \in W^2_2[0, l].$
\end{proof}
From Proposition~\ref{necessity con} it follows that asymptotic formulae~\eqref{con1} and~\eqref{con2} are necessary for $\Delta(\lambda)$ to be the characteristic function. Based on the above-mentioned, we formulate the following result.
\begin{thm} \label{conditions}
For a sequence $\{ \lambda_n \}_{n \ge 0}$ to be the spectrum of some boundary value problem~\eqref{*}--\eqref{Dirichlet} in case~\eqref{case}, the following conditions are necessary and sufficient:

1) asymptotic formulae~\eqref{sufficiency asymp} hold;

2) the function constructed via formula~\eqref{sufficiency Delta} satisfies~\eqref{con1},~\eqref{con2};

3) condition (C): the function $G_2$ constructed by formulae~\eqref{W recovery},~\eqref{gg} vanishes at $b.$
\end{thm}

{\bf Aknowledgement.} This work was financially supported by project no.~19-01-00102 of the Russian Foundation for Basic Research.

\end{document}